\documentclass[11pt,reqno]{amsart}

\usepackage[T1]{fontenc}
\usepackage{lmodern}
\usepackage{microtype}
\usepackage{amsmath,amssymb,amsfonts,mathtools}
\usepackage{enumitem}
\usepackage{xcolor}
\usepackage{aliascnt}
\usepackage[colorlinks=true,linkcolor=blue!55!black,citecolor=blue!55!black,urlcolor=blue!55!black]{hyperref}
\usepackage[nameinlink,capitalize,noabbrev]{cleveref}

\allowdisplaybreaks
\numberwithin{equation}{section}

\theoremstyle{plain}
\newtheorem{theorem}{Theorem}[section]

\newaliascnt{proposition}{theorem}
\newtheorem{proposition}[proposition]{Proposition}
\aliascntresetthe{proposition}

\newaliascnt{lemma}{theorem}
\newtheorem{lemma}[lemma]{Lemma}
\aliascntresetthe{lemma}

\newaliascnt{corollary}{theorem}
\newtheorem{corollary}[corollary]{Corollary}
\aliascntresetthe{corollary}

\theoremstyle{definition}
\newaliascnt{definition}{theorem}
\newtheorem{definition}[definition]{Definition}
\aliascntresetthe{definition}

\theoremstyle{remark}
\newaliascnt{remark}{theorem}
\newtheorem{remark}[remark]{Remark}
\aliascntresetthe{remark}

\crefname{theorem}{Theorem}{Theorems}
\Crefname{theorem}{Theorem}{Theorems}
\crefname{proposition}{Proposition}{Propositions}
\Crefname{proposition}{Proposition}{Propositions}
\crefname{lemma}{Lemma}{Lemmas}
\Crefname{lemma}{Lemma}{Lemmas}
\crefname{corollary}{Corollary}{Corollaries}
\Crefname{corollary}{Corollary}{Corollaries}
\crefname{definition}{Definition}{Definitions}
\Crefname{definition}{Definition}{Definitions}
\crefname{remark}{Remark}{Remarks}
\Crefname{remark}{Remark}{Remarks}

\newcommand{\SN}{\mathbb{S}_N}
\newcommand{\E}{\mathbb E}
\newcommand{\Pp}{\mathbb P}
\newcommand{\R}{\mathbb R}
\newcommand{\1}{\mathbf 1}
\newcommand{\dd}{\,\mathrm d}
\newcommand{\eps}{\varepsilon}
\newcommand{\bsh}{\beta_{\mathrm{sh}}}
\newcommand{\Tsh}{T_{\mathrm{sh}}}
\newcommand{\Ts}{T_{\mathrm{s}}}
\newcommand{\cN}{\mathcal C_N}
\newcommand{\MN}{\mathcal M_N}
\newcommand{\SigmaP}{\Sigma_p}
\newcommand{\PhiP}{\Phi_{p,\beta}}
\newcommand{\Om}{\Omega}
\newcommand{\Hbin}{\mathsf H}

\title[Non-shattering at and above the dynamical temperature]{Non-Shattering at and Above the Dynamical Temperature in the Spherical Pure \(p\)-Spin Model}

\author{Taegyun Kim}
\address{Department of Mathematical Sciences, Korea Advanced Institute of Science and Technology (KAIST), Daejeon, Republic of Korea}
\email{ktg11k@kaist.ac.kr}
\date{August 21, 2026}

\subjclass[2020]{60G15, 60K35, 82B44, 82D30}
\keywords{spherical spin glass, pure \(p\)-spin model, shattering, Kac--Rice formula, complexity, dynamical transition}

\hypersetup{
  pdftitle={Non-Shattering at and Above the Dynamical Temperature in the Spherical Pure p-Spin Model},
  pdfauthor={Taegyun Kim},
  pdfsubject={Geometric and marked Kac-Rice obstructions to shattering at and above the dynamical temperature},
  pdfkeywords={spherical spin glass, pure p-spin model, shattering, Kac-Rice, complexity}
}

\begin{document}

\begin{abstract}
We consider shattering in the precise fixed-overlap, critical-point-band sense formulated by Ben Arous and Jagannath for spherical pure \(p\)-spin glasses with overlap \(q\).  For every \(p\geq3\) and \(0<\beta\leq\beta_{\mathrm{sh}}(p)\), we rule out shattering whenever \(q\leq2^{-1/2}\) or \(q>\sqrt{(p-2)/(p-1)}\).  The proof combines a deterministic \(N+1\) bound for disjoint bands in the first range with a general-\(p\) sign law showing that their total marked weight has subdominant free energy in the second.  A spherical-code bound and H\"older's inequality give an additional \(q\)-dependent obstruction; in particular, they rule out every fixed overlap for \(0<\beta\leq\sqrt{\log2}\).  For \(p=3\), the first two ranges already exhaust every fixed \(q\in(0,1)\), so the landscape is not shattered at any \(T\geq T_{\mathrm{sh}}\).  For \(p\geq4\), the cases not covered by our criteria are confined to \(2^{-1/2}<q\leq\sqrt{(p-2)/(p-1)}\) and \(\sqrt{\log2}<\beta\leq\beta_{\mathrm{sh}}(p)\).  In particular, this paper partially resolves Conjecture 1 of the paper above and also suggests new methods to show non-shattering.
\end{abstract}

\maketitle

\section{Introduction}

The spherical pure \(p\)-spin model is a canonical mean-field model in which the geometry of a random high-dimensional landscape can be related to the thermodynamics and dynamics of a glassy system.  The physical picture of a higher dynamical transition in mean-field \(p\)-spin models---denoted \(T_{\mathrm{sh}}\) in the setting below---was first identified in the pioneering work of Kirkpatrick and Thirumalai, who distinguished it from the lower static glass transition and related it to broken replica symmetry \cite{KirkpatrickThirumalai1987Dynamics,KirkpatrickThirumalai1987Connections}.  Their subsequent metastable-state analysis connected the onset of such states with an ergodic-to-nonergodic transition \cite{KirkpatrickThirumalai1988Metastable}.  For the spherical formulation considered here, the statics and dynamics were analyzed in \cite{CrisantiSommers1992,CrisantiHornerSommers1993}.  Rigorous critical-point complexity and the low-temperature decomposition of the Gibbs measure into bands around deep critical points were subsequently developed in \cite{AuffingerBenArousCerny2013,Subag2017Geometry}; see also \cite{Subag2024Landscapes} for the associated band free-energy landscape and generalized TAP picture.  Building on this physical picture, Ben Arous and Jagannath formulated in \cite{BenArousJagannath2024} the precise fixed-overlap, critical-point-band notion of \emph{shattering} used here: the Gibbs free energy is asymptotically carried by exponentially many pairwise disjoint bands around critical points, while every individual band has exponentially small Gibbs mass.  Their Conjecture~1 predicts shattering for \(T_{\mathrm{s}}\leq T\leq T_{\mathrm{sh}}\) and non-shattering for \(T>T_{\mathrm{sh}}\).  Thus its printed formulation assigns the endpoint \(T=T_{\mathrm{sh}}\) to the shattered side.

We prove a partial non-shattering theorem for every spherical pure \(p\)-spin model on the high-temperature side and, for \(p=3\), a complete non-shattering theorem for all fixed overlaps and every \(T\geq T_{\mathrm{sh}}\).  We work throughout with the exact fixed-overlap, critical-point-band definition in \cite[Definition~2.3]{BenArousJagannath2024}.  The argument also gives an all-overlap result for every \(p\geq3\) when \(\beta\leq\sqrt{\log2}\), and identifies the smaller region that remains unresolved for \(p\geq4\).  More recent works use stronger or different notions of shattering; see \cite{ElAlaouiMontanariSellke2025,AuffingerElAlaouiSellke2025}.  We make no assertion about those other notions.

We write \(\Ts\) for the static replica-symmetry-breaking temperature
and \(\Tsh\) for the dynamical, or shattering, temperature, following
the convention of \cite{BenArousJagannath2024}.

Fix an integer \(p\geq3\).  Let
\[
 \SN=S^{N-1}(\sqrt N)\subset\R^N,
 \qquad
 R(x,y)=\frac{x\cdot y}{N},
\]
and let \(H_{N,p}\) be the centered Gaussian field with covariance
\begin{equation}\label{eq:covariance}
 \E H_{N,p}(x)H_{N,p}(y)=N R(x,y)^p.
\end{equation}
When \(p\) is fixed, we write \(H_N=H_{N,p}\).  Set
\begin{equation}\label{eq:generalThresholdsIntro}
 \beta_{\mathrm{sh}}(p)^2
 =\frac{(p-1)^{p-1}}{p(p-2)^{p-2}},
 \qquad
 T_{\mathrm{sh}}(p)=\beta_{\mathrm{sh}}(p)^{-1},
 \qquad
 s_*(p)=\frac{p-2}{p-1}.
\end{equation}
For \(p=3\), the formula in \cite[Section~2]{BenArousJagannath2024}
gives
\[
 \Tsh=T_{\mathrm{sh}}(3)=\frac{\sqrt3}{2},
 \qquad
 \bsh=\beta_{\mathrm{sh}}(3)=\frac{2}{\sqrt3},
 \qquad
 s_*(3)=\frac12.
\]
Our main result is the following.

\begin{theorem}[Main theorem]\label{thm:main}
Fix \(p\geq3\), \(0<\beta\leq\beta_{\mathrm{sh}}(p)\),
\(E\in\R\), and \(r\geq0\).  The spherical pure \(p\)-spin
free-energy landscape is not \((E,q,r)\)-shattered whenever
\begin{equation}\label{eq:mainRange}
 q\in\left(0,\frac1{\sqrt2}\right]
 \cup\left(\sqrt{\frac{p-2}{p-1}},1\right).
\end{equation}
Equivalently, the inverse-temperature assumption is
\(T\geq T_{\mathrm{sh}}(p)\).

For \(q>2^{-1/2}\), the same conclusion also holds whenever
\begin{equation}\label{eq:mainCodeCriterion}
 \beta^2<
 -2\mathcal R_{\mathrm{LP}}(2q^2-1)-\log(1-q^2),
\end{equation}
where the spherical-code rate \(\mathcal R_{\mathrm{LP}}\) is defined
in \eqref{eq:RLP} below.

In particular, when \(p=3\), the two intervals in
\eqref{eq:mainRange} exhaust \(q\in(0,1)\).  Hence the spherical pure
\(3\)-spin landscape is not \((E,q,r)\)-shattered for any fixed
\(E\in\R\), \(q\in(0,1)\), and \(r\geq0\), throughout
\(T\geq\Tsh\).

Consequently, for every \(p\geq3\), the landscape is not
\((E,q,r)\)-shattered for any fixed \(E\in\R\), \(q\in(0,1)\), and
\(r\geq0\), whenever
\begin{equation}\label{eq:hotAllOverlap}
 0<\beta\leq\sqrt{\log2},
 \qquad\text{equivalently}\qquad
 T\geq(\log2)^{-1/2}.
\end{equation}
\end{theorem}

\begin{corollary}\label{cor:conjecture}
For \(p=3\), the non-shattering assertion in Conjecture~1 of
\cite{BenArousJagannath2024} for \(T>T_{\mathrm{sh}}\) holds and extends
to \(T=T_{\mathrm{sh}}\).  Consequently, both the inclusion of the
endpoint in the shattered regime of Conjecture~1 and the endpoint
assertion of \cite[Theorem~8.4]{BenArousJagannath2024} are false as
printed.
\end{corollary}

For \(p=3\), the specialization of \cref{thm:main} proves the
conjectured non-shattering statement for \(T>T_{\mathrm{sh}}\) and
extends it to equality.  The new issue is equality at the endpoint,
and the discrepancy there is already visible in the published proof.  Although
\cite[Theorem~8.4]{BenArousJagannath2024} includes \(T=\Tsh\), its joint
proof with Theorem~2.5 begins under the strict assumption
\(\bsh<\beta<\Ts^{-1}\); see the proof of Theorems~2.5 and~8.4 in
\cite[Section~5]{BenArousJagannath2024}.  The discussion specialized to
\(p=3\) again uses \(\beta>\bsh\) to place the relevant overlap strictly
above the geometric threshold; see \cite[Section~8]{BenArousJagannath2024}.

The proof is organized around a variational statement valid for general
\(p\).  For a critical point \(x\), write \(u=-H_N(x)/N\), set
\(s=q^2\), and let \(\Phi_{p,\beta}\) be the marked exponent defined in
\eqref{eq:PhiGeneral}.  We prove that, for every
\(0<\beta\leq\beta_{\mathrm{sh}}(p)\),
\begin{equation}\label{eq:introSignLaw}
 \operatorname{sgn}\!\left(
  \sup_{u\in\R}\Phi_{p,\beta}(u,\sqrt s)-\frac{\beta^2}{2}
 \right)
 =\begin{cases}
  1,&0\leq s<s_*(p),\\
  0,&s=s_*(p),\\
  -1,&s_*(p)<s<1.
 \end{cases}
\end{equation}
When \(p=3\), the transition point is \(s_*=1/2\).  Pairwise
disjointness gives at most \(N+1\) centers for \(s\leq1/2\), whereas
the marked Kac--Rice bound gives a strict free-energy loss for
\(s>1/2\).  These alternatives cover every fixed overlap and every
\(0<\beta\leq\bsh\).

For \(p\geq4\), the transition point satisfies \(s_*(p)>1/2\).  The
interval \(1/2<s\leq s_*(p)\) is therefore controlled by neither of
these two arguments: the deterministic bound no longer applies, while
the marked exponent has no strict negative gap.  A third argument,
based on the asymptotic linear-programming bound for spherical codes
and H\"older's inequality, supplies a further \(q\)-dependent
obstruction and settles this entire interval when
\(\beta\leq\sqrt{\log2}\).  Thus the cases not covered by our criteria
are confined to \(\sqrt{\log2}<\beta\leq\beta_{\mathrm{sh}}(p)\).

The random-landscape input uses only one-point Kac--Rice analysis.  The extra Gibbs mark is handled by Gaussian regression.  Conditional on the value and gradient at a prospective critical point, the Hamiltonian at a point on a latitude has the standard conditional mean and variance.  Its covariance with the conditional Hessian has rank at most one.  Exponential tilting therefore changes the conditional Hessian only by a bounded rank-one perturbation.  A cofactor expansion, combined with the expected absolute-determinant asymptotics in \cite[Corollary~1.3]{BenArousBourgadeMcKenna2022}, shows that this perturbation does not alter the exponential determinant rate.  This yields the marked exponent in \eqref{eq:PhiGeneral}.

The rest of the paper is organized as follows.  We first set notation
and recall the shattering definition, then prove the deterministic
small-overlap obstruction and a spherical-code entropy bound.  We next
prove the marked Kac--Rice bound, establish the general-\(p\) sign law,
and combine it with the small-overlap geometry
to prove \cref{thm:main}.  We then specialize to \(p=3\) to record the
complete all-overlap conclusion and a sharper fourth-order endpoint
tangency.  The final section discusses the endpoint claim and the
relation with other notions of shattering.

\section{Model, free energy, and shattering}\label{sec:model}

Let \(\mu_N\) be the normalized uniform measure on \(\SN\).  For a Borel set \(A\subseteq\SN\), define
\begin{equation}\label{eq:partition}
 Z_{N,\beta}(A)=\int_A e^{-\beta H_N(\sigma)}\,\mu_N(\dd\sigma),
 \qquad
 F_{N,\beta}(A)=\frac1N\log Z_{N,\beta}(A).
\end{equation}
We abbreviate
\[
 Z_{N,\beta}=Z_{N,\beta}(\SN),
 \qquad
 F_N(\beta)=F_{N,\beta}(\SN).
\]
For \(x\in\SN\), \(q\in(-1,1)\), and \(\eta>0\), let
\begin{equation}\label{eq:band}
B(x,q,\eta)=\bigl\{\sigma\in\SN:\lvert R(x,\sigma)-q\rvert\leq\eta\bigr\}.
\end{equation}
For fixed \(x\in\SN\), the overlap of a \(\mu_N\)-uniform point with
\(x\) has density
\begin{equation}\label{eq:overlapdensity}
 \nu_N(t)=\frac{\Gamma(N/2)}
 {\sqrt\pi\,\Gamma((N-1)/2)}(1-t^2)^{(N-3)/2},
 \qquad -1<t<1.
\end{equation}
For an interval \(I\subset\R\), write
\begin{equation}\label{eq:criticalset}
 \cN(I)=\left\{x\in\SN:\nabla H_N(x)=0,\ \frac{H_N(x)}N\in I\right\}.
\end{equation}

We record the definition from \cite{BenArousJagannath2024} in the form used below.

\begin{definition}[Ben Arous--Jagannath shattering]\label{def:shattering}
Fix \(T>0\), \(E\in\R\), \(q\in(0,1)\), and \(r\geq0\), and set \(\beta=T^{-1}\).  The free-energy landscape is \((E,q,r)\)-shattered if there exist constants \(c,c'>0\), positive sequences \(\eps_N,\eta_N,\delta_N\to0\), and random sets
\[
 A_N\subseteq \cN([-E-\eps_N,-E+\eps_N])
\]
such that, with probability tending to one,
\begin{enumerate}[label=\textup{(\arabic*)},leftmargin=2.2em]
 \item \(N^{-1}\log\lvert A_N\rvert\geq c\);
 \item the bands \(B(x,q,\eta_N)\), \(x\in A_N\), are pairwise disjoint, and \(R(x,y)<r\) for distinct \(x,y\in A_N\);
 \item every band is exponentially subdominant:
 \[
 F_N(\beta)-F_{N,\beta}(B(x,q,\eta_N))>c';
 \]
 \item the union has the full limiting free energy:
 \[
 F_N(\beta)-F_{N,\beta}\left(\bigcup_{x\in A_N}B(x,q,\eta_N)\right)\xrightarrow{\Pp}0.
 \]
\end{enumerate}
The three vanishing sequences are deterministic and positive, the finite random sets \(A_N\) are measurable, conditions \textup{(1)}--\textup{(3)} hold jointly with probability tending to one, and condition \textup{(4)} is a convergence-in-probability statement.
\end{definition}

\begin{remark}[Fixed overlap parameter]\label{rem:fixedq}
As in the published definition, \(E\), \(q\), and \(r\) are fixed
independently of \(N\).  In particular, \cref{thm:main} does not address
modified notions in which \(q=q_N\).  Sequences approaching either
\(2^{-1/2}\) or \(\sqrt{(p-2)/(p-1)}\) lie outside the definition
considered here; for \(p=3\), the two boundary values coincide.
\end{remark}

Throughout the high-temperature range of \cref{thm:main}, the limiting
free energy is replica symmetric.

\begin{lemma}\label{lem:fullFE}
Fix \(p\geq3\) and \(0<\beta\leq\beta_{\mathrm{sh}}(p)\).  Then
\begin{equation}\label{eq:fullFE}
 F_N(\beta)\xrightarrow{\Pp}\frac{\beta^2}{2}.
\end{equation}
\end{lemma}

\begin{proof}
For a spherical model with covariance function \(\xi(t)=t^p\), the
standard replica-symmetry criterion states that the limiting free
energy equals \(\beta^2\xi(1)/2\) provided
\[
 g_\beta(t)=\beta^2t^p+\log(1-t)+t\leq0,
 \qquad 0\leq t<1;
\].
The spherical Parisi formula was proved for even mixtures in
\cite{Talagrand2006} and for general mixtures, including odd \(p\), in
\cite{Chen2013}; the above criterion in the present pure setting is
recorded in \cite[Equation~(3.8)]{BenArousJagannath2024}.  At
\(\beta=\beta_{\mathrm{sh}}(p)\),
\[
 g_\beta'(t)
 =\frac{t}{1-t}
  \left(\beta^2p t^{p-2}(1-t)-1\right)\leq0.
\]
Indeed, \(t^{p-2}(1-t)\) is maximized at
\(t=(p-2)/(p-1)\), where the parenthesized expression vanishes.
Since \(g_\beta(0)=0\), this gives \(g_\beta(t)\leq0\) at the endpoint.
The same inequality holds for every smaller \(\beta\).  The criterion
and concentration of the spherical free energy yield \eqref{eq:fullFE}.
\end{proof}

\section{A deterministic obstruction for small band overlap}\label{sec:geometry}

For \(x\in\SN\) and \(q\in(-1,1)\), define the exact latitude
\begin{equation}\label{eq:latitude}
 L(x,q)=\{\sigma\in\SN:R(x,\sigma)=q\}.
\end{equation}
Clearly \(L(x,q)\subset B(x,q,\eta)\) for every \(\eta>0\).

\begin{lemma}[Intersection criterion for exact latitudes]\label{lem:latitude}
Assume \(N\geq3\).  Let \(x,y\in\SN\) be distinct and set
\(\rho=R(x,y)\).  For \(q\in(0,1)\),
\begin{equation}\label{eq:latcriterion}
 L(x,q)\cap L(y,q)\neq\varnothing
 \quad\Longleftrightarrow\quad
 \rho\geq 2q^2-1.
\end{equation}
Consequently, if \(B(x,q,\eta)\cap B(y,q,\eta)=\varnothing\) for some \(\eta>0\), then
\begin{equation}\label{eq:centeroverlap}
 R(x,y)<2q^2-1.
\end{equation}
\end{lemma}

\begin{proof}
A point \(\sigma\in\SN\) with
\[
 R(x,\sigma)=R(y,\sigma)=q
\]
exists if and only if the Gram matrix
\[
 G=\begin{pmatrix}
 1&\rho&q\\
 \rho&1&q\\
 q&q&1
 \end{pmatrix}
\]
is positive semidefinite.  Indeed, necessity is immediate; conversely, a
positive semidefinite \(3\times3\) Gram matrix has a realization in
\(\R^N\) because \(N\geq3\), and an orthogonal map can match its first two
vectors to \(x/\sqrt N\) and \(y/\sqrt N\).  Since \(x\neq y\), we have
\(\rho<1\), and
\[
 \det G=(1-\rho)(1+\rho-2q^2).
\]
The principal \(2\times2\) minors are nonnegative, so \(G\succeq0\) is
equivalent to \(1+\rho-2q^2\geq0\), proving \eqref{eq:latcriterion}.  The
last assertion follows from \(L(x,q)\subset B(x,q,\eta)\).
\end{proof}

We use the following classical simplex bound for spherical codes, due
to Rankin \cite[Theorem~1]{Rankin1955}; a proof is included for
convenience.

\begin{lemma}\label{lem:obtuse}
Let \(v_1,\dots,v_M\in\R^N\) be unit vectors satisfying
\[
 v_i\cdot v_j<0\qquad(i\neq j).
\]
Then \(M\leq N+1\).
\end{lemma}

\begin{proof}
Suppose \(M>N+1\).  The points \(v_1,\dots,v_M\) are affinely dependent, so there are real numbers \(a_i\), not all zero, such that
\[
 \sum_{i=1}^M a_i v_i=0,
 \qquad
 \sum_{i=1}^M a_i=0.
\]
There are both positive and negative coefficients.  After splitting the indices into \(P=\{i:a_i>0\}\) and \(M_-=\{j:a_j<0\}\), and normalizing, we obtain positive coefficients \(\alpha_i,\beta_j\) with each family summing to one and
\[
 z:=\sum_{i\in P}\alpha_i v_i
 =\sum_{j\in M_-}\beta_j v_j.
\]
Taking the inner product of the two representations gives
\[
 \lVert z\rVert^2
 =\sum_{i\in P}\sum_{j\in M_-}\alpha_i\beta_j(v_i\cdot v_j)<0,
\]
a contradiction.
\end{proof}

\begin{proposition}[No positive complexity for \(q\leq2^{-1/2}\)]\label{prop:smallq}
Assume \(N\geq3\), and fix \(q\in(0,2^{-1/2}]\) and \(\eta>0\).  If
\(A\subseteq\SN\) and the bands
\[
 \{B(x,q,\eta):x\in A\}
\]
are pairwise disjoint, then \(\lvert A\rvert\leq N+1\).  Hence condition
\textup{(1)} in \cref{def:shattering} cannot hold.
\end{proposition}

\begin{proof}
By \cref{lem:latitude}, distinct centers satisfy
\[
 R(x,y)<2q^2-1\leq0.
\]
Applying \cref{lem:obtuse} to the unit vectors \(x/\sqrt N\), \(x\in A\), gives \(\lvert A\rvert\leq N+1\).
\end{proof}

We next quantify what remains possible once the overlap threshold in
\cref{prop:smallq} becomes positive.  Let \(A_N^{\mathrm{sph}}(a)\)
denote the largest cardinality of a subset of the unit sphere
\(S^{N-1}(1)\) whose distinct elements have inner product at most
\(a\).  Write
\[
 \Hbin(x)=-x\log x-(1-x)\log(1-x),
 \qquad 0\leq x\leq1,
\]
with the convention \(0\log0=0\), and, for \(0\leq a<1\), define
\begin{equation}\label{eq:RLP}
 \mathcal R_{\mathrm{LP}}(a)
 =\frac{1+t}{2t}\,
 \Hbin\left(\frac{1-t}{1+t}\right),
 \qquad t=\sqrt{1-a^2}.
\end{equation}
The asymptotic linear-programming bound of Kabatiansky and Levenshtein
\cite[Theorem~4 and Equation~(54)]{KabatianskyLevenshtein1978}, after
converting their base-two logarithms to natural logarithms, states that
for every fixed \(a\in(0,1)\),
\begin{equation}\label{eq:KLcodeBound}
 \limsup_{N\to\infty}\frac1N\log A_N^{\mathrm{sph}}(a)
 \leq\mathcal R_{\mathrm{LP}}(a).
\end{equation}

\begin{lemma}[A comparison for the spherical-code rate]
\label{lem:RLPcomparison}
For \(0\leq a<1\),
\begin{equation}\label{eq:RLPcomparison}
 \mathcal R_{\mathrm{LP}}(a)
 \leq-\frac12\log(1-a),
\end{equation}
with equality if and only if \(a=0\).
\end{lemma}

\begin{proof}
The assertion is immediate when \(a=0\).  For \(0<a<1\), put
\[
 x=\frac{1-t}{1+t},\qquad y=\sqrt{x}.
\]
Then \(0<y<1\) and
\[
 t=\frac{1-y^2}{1+y^2},\qquad
 a=\frac{2y}{1+y^2},\qquad
 \frac{1+t}{2t}=\frac1{1-y^2}.
\]
Consequently,
\begin{align*}
 \mathcal R_{\mathrm{LP}}(a)
 &=-\frac{2y^2}{1-y^2}\log y-\log(1-y)-\log(1+y),\\
 -\frac12\log(1-a)
 &=-\log(1-y)+\frac12\log(1+y^2).
\end{align*}
Now \(-2y\log y<1-y^2\) for \(0<y<1\), because
\(g(y)=1-y^2+2y\log y\) satisfies \(g(1)=0\) and
\(g'(y)=2(1-y+\log y)<0\), so \(g(y)>g(1)=0\).  Hence
\[
 -\frac{2y^2}{1-y^2}\log y<y.
\]
Moreover, if
\[
 k(y)=\log(1+y)+\frac12\log(1+y^2)-y,
\]
then \(k(0)=0\) and
\[
 k'(y)=\frac{y^2(1-y)}{(1+y)(1+y^2)}>0.
\]
Combining the last two inequalities proves
\eqref{eq:RLPcomparison}, strictly when \(a>0\).
\end{proof}

\begin{proposition}[Entropy loss for a union of disjoint bands]
\label{prop:bandEntropy}
Fix \(q\in(2^{-1/2},1)\), let \(\eta_N\downarrow0\), and let
\(A_N\subseteq\SN\) be any possibly random set such that the bands
\(B(x,q,\eta_N)\), \(x\in A_N\), are pairwise disjoint.  There is a
deterministic sequence \(e_N=e_N(q,(\eta_k)_k)\to0\), independent of
\(A_N\) and of the disorder, such that
\begin{equation}\label{eq:bandEntropy}
 \frac1N\log
 \mu_N\left(\bigcup_{x\in A_N}B(x,q,\eta_N)\right)
 \leq h(q)+e_N,
 \qquad h(q)<-\frac12\log2,
\end{equation}
where
\begin{equation}\label{eq:hq}
 h(q)=\mathcal R_{\mathrm{LP}}(2q^2-1)
       +\frac12\log(1-q^2).
\end{equation}
\end{proposition}

\begin{proof}
Put \(q_N=q-\eta_N\).  For all large \(N\), \(q_N>0\), and
\(L(x,q_N)\subseteq B(x,q,\eta_N)\).  Applying
\cref{lem:latitude} at \(q_N\) therefore shows that disjointness
implies
\[
 R(x,y)<2q_N^2-1\qquad(x\ne y).
\]
Since \(q_N<q\), the normalized centers therefore form a spherical
code whose distinct inner products are at most \(2q^2-1\).  Also, by
\eqref{eq:overlapdensity}, every band has
the same measure and
\[
 \mu_N(B(x,q,\eta_N))
 \leq V_N(q_N):=\int_{q_N}^1\nu_N(t)\,\dd t.
\]
It follows that
\begin{equation}\label{eq:codeCapProduct}
 \mu_N\left(\bigcup_{x\in A_N}B(x,q,\eta_N)\right)
 \leq A_N^{\mathrm{sph}}(2q^2-1)V_N(q_N).
\end{equation}
For every fixed sufficiently small \(\delta>0\), the right-hand side
is at most
\[
 A_N^{\mathrm{sph}}(2q^2-1)V_N(q-\delta)
\]
for all large \(N\).  The code bound \eqref{eq:KLcodeBound} and the
standard spherical-cap asymptotic
\[
 \lim_{N\to\infty}\frac1N\log V_N(v)
 =\frac12\log(1-v^2),\qquad 0<v<1,
\]
therefore give, after \(\delta\downarrow0\), the first inequality in
\eqref{eq:bandEntropy}.  More explicitly, if \(b_N\) is the logarithm
of the deterministic right-hand side of \eqref{eq:codeCapProduct},
divided by \(N\) and minus \(h(q)\), then \(\limsup_N b_N\leq0\).
Thus
\[
 e_N=\max\{b_N,0\}+N^{-1}
\]
tends to zero and is uniform over all choices of \(A_N\).  Finally,
\cref{lem:RLPcomparison} applied to
\(a=2q^2-1>0\), together with
\(1-a=2(1-q^2)\), gives
\[
 h(q)<-\frac12\log\bigl(2(1-q^2)\bigr)
       +\frac12\log(1-q^2)
 =-\frac12\log2.
\]
\end{proof}

\begin{proposition}[Spherical-code--H\"older obstruction]
\label{prop:hotAllOverlap}
Fix \(p\geq3\), \(0<\beta\leq\beta_{\mathrm{sh}}(p)\), \(E\in\R\),
\(r\geq0\), and \(q\in(2^{-1/2},1)\).  Set
\begin{equation}\label{eq:cLPq}
 c_{\mathrm{LP}}(q)
 =-\mathcal R_{\mathrm{LP}}(2q^2-1)
  -\frac12\log(1-q^2).
\end{equation}
If
\begin{equation}\label{eq:qDependentCodeCriterion}
 \beta^2<2c_{\mathrm{LP}}(q),
\end{equation}
then the landscape is not \((E,q,r)\)-shattered.

Consequently, for every \(p\geq3\), \(E\in\R\), and \(r\geq0\),
the landscape is not \((E,q,r)\)-shattered for any fixed
\(q\in(0,1)\) whenever \(0<\beta\leq\sqrt{\log2}\).
\end{proposition}

\begin{proof}
Suppose that a shattering family exists and write
\[
 U_N=\bigcup_{x\in A_N}B(x,q,\eta_N),
 \qquad m_N=\frac1N\log\mu_N(U_N).
\]
Let \(D_N\) be the event that these bands are pairwise disjoint; by
condition~\textup{(2)} in \cref{def:shattering},
\(\Pp(D_N)\to1\).
For every \(\lambda>1\), H\"older's inequality gives the pathwise bound
\begin{align}
 Z_{N,\beta}(U_N)
 &\leq \mu_N(U_N)^{1-1/\lambda}
       Z_{N,\lambda\beta}^{1/\lambda},\notag\\
 F_{N,\beta}(U_N)
 &\leq\left(1-\frac1\lambda\right)m_N
       +\frac1\lambda F_N(\lambda\beta).
 \label{eq:HolderRestricted}
\end{align}
Set \(c_q=c_{\mathrm{LP}}(q)=-h(q)\).  By
\eqref{eq:qDependentCodeCriterion}, we may choose \(\lambda>1\) such
that \(\lambda\beta^2<2c_q\).

Choose a fixed \(c\) satisfying \(\lambda\beta^2/2<c<c_q\).  By
\cref{prop:bandEntropy}, for all sufficiently large \(N\),
\(D_N\) implies \(m_N\leq-c\).
For every \(\theta>0\), Gaussian integration gives
\[
 \E Z_{N,\theta}=e^{N\theta^2/2}.
\]
Markov's inequality, together with \cref{lem:fullFE} at \(\beta\),
therefore gives deterministic sequences \(a_N,b_N\downarrow0\) and
events \(G_N\), with \(\Pp(G_N)\to1\), on which
\[
 F_N(\lambda\beta)\leq\frac{\lambda^2\beta^2}{2}+a_N,
 \qquad
 F_N(\beta)\geq\frac{\beta^2}{2}-b_N.
\]
For instance, one may take \(a_N=N^{-1/2}\), for which the Markov
failure probability is at most \(e^{-\sqrt N}\); the sequence \(b_N\)
follows by a diagonal choice from convergence in probability.
On \(D_N\cap G_N\), \eqref{eq:HolderRestricted} therefore gives
\[
 F_{N,\beta}(U_N)
 \leq-\left(1-\frac1\lambda\right)c
       +\frac{\lambda\beta^2}{2}+\frac{a_N}{\lambda}.
\]
The deterministic part of the right-hand side lies below
\(\beta^2/2\) by the positive amount
\[
 (\lambda-1)\left(\frac c\lambda-\frac{\beta^2}{2}\right).
\]
It follows that a fixed positive lower bound on
\(F_N(\beta)-F_{N,\beta}(U_N)\) holds with probability tending to
one, contradicting condition~\textup{(4)} in \cref{def:shattering}.

Finally, \cref{prop:bandEntropy} gives
\(c_{\mathrm{LP}}(q)>\frac12\log2\) for every
\(q>2^{-1/2}\).  Thus \(\beta^2\leq\log2\) implies the strict
criterion \eqref{eq:qDependentCodeCriterion} for every such fixed
\(q\), while \cref{prop:smallq} treats \(q\leq2^{-1/2}\).  This range
is contained in the assumed dynamical replica-symmetric range: with \(n=p-2\),
the binomial theorem gives
\[
 \beta_{\mathrm{sh}}(p)^2
 =\frac{n+1}{n+2}\left(1+\frac1n\right)^n
 \geq\frac43>\log2.
\]
\end{proof}

\section{A marked Kac--Rice upper bound}\label{sec:KR}

We next control the total partition-function mass of all bands centered
at critical points.  Throughout this section, \(p\in\{3,4,\ldots\}\)
is fixed and
\[
 \E H_N(x)H_N(y)=N R(x,y)^p.
\]
Put \(n=N-1\).  We use the normalization in which the independent
upper-triangular entries of the \(n\times n\) GOE matrix \(G_n\) satisfy
\begin{equation}\label{eq:GOEnormalization}
 \E(G_n)_{ij}^2=\frac1n\quad(i<j),
 \qquad
 \E(G_n)_{ii}^2=\frac2n.
\end{equation}
Thus its empirical spectral law converges to the semicircle law on
\([-2,2]\).  Define
\begin{equation}\label{eq:KNexact}
 \alpha_{N,p}=\sqrt{\frac{p(p-1)(N-1)}N},
 \qquad
 K_N(u)=\alpha_{N,p}G_{N-1}+puI_{N-1}.
\end{equation}

Let \(\mu_{\mathrm{sc}}\) be the semicircle law
\[
 \mu_{\mathrm{sc}}(\dd\lambda)
 =\frac{\sqrt{4-\lambda^2}}{2\pi}
  \1_{\{|\lambda|\leq2\}}\dd\lambda,
\]
and define its logarithmic potential
\begin{equation}\label{eq:Omega}
 \Om(x)=\int_{-2}^2\log\lvert x-\lambda\rvert\,
 \mu_{\mathrm{sc}}(\dd\lambda).
\end{equation}
The one-point annealed critical-point complexity is
\begin{equation}\label{eq:Sigma}
 \SigmaP(u)
 =\frac12+\frac12\log(p-1)-\frac{u^2}{2}
  +\Om\left(\sqrt{\frac{p}{p-1}}\,u\right).
\end{equation}
For \(t\in(-1,1)\), set
\begin{equation}\label{eq:vpt}
 v_p(t)=1-t^{2p}-p t^{2p-2}(1-t^2),
 \qquad
 a_p(t)=p(p-1)t^{p-2}(1-t^2),
\end{equation}
and, for \(\beta>0\), define
\begin{equation}\label{eq:PhiGeneral}
 \PhiP(u,t)
 =\SigmaP(u)+\beta t^p u
  +\frac12\log(1-t^2)
  +\frac{\beta^2}{2}v_p(t).
\end{equation}

For \(q\in(0,1)\) and \(\eta>0\), let
\begin{equation}\label{eq:markedmass}
 \MN(q,\eta)
 =\sum_{x:\,\nabla H_N(x)=0} Z_{N,\beta}(B(x,q,\eta)).
\end{equation}
The sum is finite almost surely by the standard nondegeneracy of the
pure spherical model; this is also part of the Kac--Rice setup in
\cite[Section~3]{AuffingerBenArousCerny2013}.

For Borel sets \(A\subset\R\) and \(I\subset(-1,1)\), write
\begin{align}
 \mathcal K_N(A,I)
 ={}&\E\sum_{x:\,\nabla H_N(x)=0}
 \1_{\{-H_N(x)/N\in A\}}\notag\\
 &\quad\times
 \int_{\{\sigma:R(x,\sigma)\in I\}}
 e^{-\beta H_N(\sigma)}\,\mu_N(\dd\sigma).
 \label{eq:calKdefinition}
\end{align}

\begin{lemma}[Finite-dimensional critical-point law]
\label{lem:finitecritical}
For every \(x\in\SN\) and \(u\in\R\), in an orthonormal tangent frame,
\begin{equation}\label{eq:conditionalHessian}
 \mathcal L\left(\nabla^2H_N(x)\mid
 H_N(x)=-Nu,\ \nabla H_N(x)=0\right)
 =\mathcal L(K_N(u)).
\end{equation}
Equivalently, the conditional mean is \(puI_{N-1}\), and the centered
conditional covariance is
\begin{equation}\label{eq:conditionalHessianCov}
 \operatorname{Cov}\bigl((\nabla^2H_N)_{ij},(\nabla^2H_N)_{k\ell}
 \mid H_N,\nabla H_N\bigr)
 =\frac{p(p-1)}N
 (\delta_{ik}\delta_{j\ell}+\delta_{i\ell}\delta_{jk}).
\end{equation}
\end{lemma}

\begin{proof}
This follows by differentiating the covariance kernel.  A direct
derivation, including the spherical curvature term, is given in
Appendix~\ref{app:regression}; compare
\cite[Lemma~3.2]{AuffingerBenArousCerny2013}.
\end{proof}

Define the expected critical-value density \(\rho_N\) by
\begin{equation}\label{eq:rhoDefinition}
 \E\sum_{x:\,\nabla H_N(x)=0}
 f\left(-\frac{H_N(x)}N\right)
 =\int_{\R}f(u)\rho_N(u)\,\dd u
\end{equation}
for every nonnegative Borel function \(f\).

\begin{lemma}[Exact one-point Kac--Rice density]\label{lem:exactrho}
Let
\begin{equation}\label{eq:kappa}
 \begin{aligned}
  \omega_N&=\operatorname{Vol}(S^{N-1}(\sqrt N))
  =\frac{2\pi^{N/2}N^{(N-1)/2}}{\Gamma(N/2)},\\
  \kappa_{N,p}&=\omega_N(2\pi p)^{-(N-1)/2}
  \sqrt{\frac N{2\pi}}.
 \end{aligned}
\end{equation}
Then
\begin{equation}\label{eq:rhoExact}
 \rho_N(u)=\kappa_{N,p}e^{-Nu^2/2}
 \E\lvert\det K_N(u)\rvert,
\end{equation}
and
\begin{equation}\label{eq:kappaLimit}
 \frac1N\log\kappa_{N,p}\longrightarrow
 \frac12-\frac12\log p.
\end{equation}
\end{lemma}

\begin{proof}
The gradient has covariance \(pI_{N-1}\) and is independent of the
field value.  The one-point Kac--Rice formula, the change of variables
\(H_N(x)=-Nu\), and \cref{lem:finitecritical} therefore give
\eqref{eq:rhoExact}; see also
\cite[Lemma~3.1 and Section~3]{AuffingerBenArousCerny2013}.  Stirling's
formula gives \eqref{eq:kappaLimit}.
\end{proof}

Let \(K_N(u)^{(1)}\) be the principal minor obtained by deleting the
first row and column, and define the finite-\(N\) determinant rates
\begin{align}
 \mathcal L_{N,p}(u)
 &=\frac1N\log\E|\det K_N(u)|,
 \label{eq:finiteFullDetRate}\\
 \mathcal L_{N,p}^{(1)}(u)
 &=\frac1N\log\E|\det K_N(u)^{(1)}|.
 \label{eq:finiteMinorDetRate}
\end{align}

\begin{lemma}[Uniform GOE determinant asymptotics]
\label{lem:uniformGOEdet}
For every \(L>0\),
\begin{align}
 \sup_{|u|\leq L}\left|
 \mathcal L_{N,p}(u)
 -\left[\frac12\log(p(p-1))
 +\Om\left(\sqrt{\frac p{p-1}}\,u\right)\right]
 \right|&\longrightarrow0,\label{eq:fullDetUniform}\\
 \sup_{|u|\leq L}\left|
 \mathcal L_{N,p}^{(1)}(u)
 -\left[\frac12\log(p(p-1))
 +\Om\left(\sqrt{\frac p{p-1}}\,u\right)\right]
 \right|&\longrightarrow0.\label{eq:minorDetUniform}
\end{align}
In particular, the limiting determinant rate is well defined by
\begin{equation}\label{eq:detRate}
 \mathcal L_p(u)
 :=\lim_{N\to\infty}\mathcal L_{N,p}(u)
 =\lim_{N\to\infty}\mathcal L_{N,p}^{(1)}(u)
 =\frac12\log(p(p-1))
 +\Om\left(\sqrt{\frac p{p-1}}\,u\right).
\end{equation}
Consequently,
\begin{equation}\label{eq:rhoSigma}
 \sup_{|u|\leq L}\left|
 \frac1N\log\rho_N(u)-\SigmaP(u)
 \right|\longrightarrow0.
\end{equation}
\end{lemma}

\begin{proof}
Writing
\[
 K_N(u)=\alpha_{N,p}
 \left(G_{N-1}-z_N(u)I_{N-1}\right),
 \qquad z_N(u)=-\frac{pu}{\alpha_{N,p}},
\]
we have \(z_N(u)\to-\sqrt{p/(p-1)}\,u\) locally uniformly.  The
expected absolute-determinant asymptotics for Wigner matrices
\cite[Corollary~1.3]{BenArousBourgadeMcKenna2022} hold locally uniformly
in the spectral parameter, including at the spectral edges.  They give
\eqref{eq:fullDetUniform}.  Moreover,
\begin{equation}\label{eq:minorLaw}
 G_{N-1}^{(1)}\stackrel d=
 \sqrt{\frac{N-2}{N-1}}\,G_{N-2},
\end{equation}
so the same result gives \eqref{eq:minorDetUniform}.  Finally,
\eqref{eq:rhoExact}, \eqref{eq:kappaLimit}, and \eqref{eq:detRate}
give \eqref{eq:rhoSigma}, which also recovers the standard complexity
formula \eqref{eq:Sigma} of \cite{AuffingerBenArousCerny2013}.
\end{proof}

Fix \(x\in\SN\), and choose \(\sigma\in\SN\) with
\(R(x,\sigma)=t\).  We use the shorthand
\[
 \mathcal E_{x,u}=\{H_N(x)=-Nu,\ \nabla H_N(x)=0\}.
\]
Gaussian regression gives
\begin{align}
 \E[H_N(\sigma)\mid\mathcal E_{x,u}]
 &=-Nt^pu,\label{eq:condmean}\\
 \operatorname{Var}(H_N(\sigma)\mid\mathcal E_{x,u})
 &=Nv_p(t).\label{eq:condvar}
\end{align}
In geodesic normal coordinates whose first tangent direction points
toward \(\sigma\), one also has the rank-at-most-one covariance
\begin{equation}\label{eq:rankonecov}
 \operatorname{Cov}(\nabla^2H_N(x),H_N(\sigma)\mid\mathcal E_{x,u})
 =a_p(t)e_1e_1^{\mathsf T}.
\end{equation}
These identities are verified in Appendix~\ref{app:regression}.

\begin{lemma}[Exact marked Kac--Rice identity]\label{lem:exactmarkedKR}
For Borel sets \(A\subset\R\) and \(I\subset(-1,1)\),
\begin{align}
 \mathcal K_N(A,I)
 ={}&\kappa_{N,p}\int_A e^{-Nu^2/2}\int_I\nu_N(t)
 \exp\left\{N\beta t^pu+\frac{N\beta^2}{2}v_p(t)\right\}\notag\\
 &\quad\times
 \E\left|\det\left(K_N(u)-\beta a_p(t)e_1e_1^{\mathsf T}\right)\right|
 \,\dd t\,\dd u.
 \label{eq:exactMarkedKR}
\end{align}
Equivalently, with
\begin{equation}\label{eq:DNdefinition}
 D_N(u,t)=
 \frac{\E|\det(K_N(u)-\beta a_p(t)e_1e_1^{\mathsf T})|}
 {\E|\det K_N(u)|},
\end{equation}
one has
\begin{equation}\label{eq:exactMarkedRatio}
 \mathcal K_N(A,I)=\int_A\rho_N(u)\int_I\nu_N(t)
 e^{N\beta t^pu+N\beta^2v_p(t)/2}D_N(u,t)\,\dd t\,\dd u.
\end{equation}
\end{lemma}

\begin{proof}
If \((X,Y)\) is jointly Gaussian and \(F\geq0\), Gaussian exponential
tilting gives
\[
 \E[F(X)e^{-\beta Y}]
 =e^{-\beta\E Y+\beta^2\operatorname{Var}(Y)/2}
 \E[F(X-\beta\operatorname{Cov}(X,Y))].
\]
Apply this identity conditionally under \(\mathcal E_{x,u}\) with
\(X=\nabla^2H_N(x)\), \(Y=H_N(\sigma)\), and
\(F(X)=|\det X|\).  Equations \eqref{eq:condmean}--\eqref{eq:rankonecov}
give
\begin{align}
 &\E[|\det\nabla^2H_N(x)|e^{-\beta H_N(\sigma)}
 \mid\mathcal E_{x,u}]\notag\\
 &\quad=e^{N\beta t^pu+N\beta^2v_p(t)/2}
 \E|\det(K_N(u)-\beta a_p(t)e_1e_1^{\mathsf T})|.
 \label{eq:tiltedDet}
\end{align}
The marked one-point Kac--Rice formula, isotropy, and
\eqref{eq:overlapdensity} now yield \eqref{eq:exactMarkedKR}.
Substituting \eqref{eq:rhoExact} gives \eqref{eq:exactMarkedRatio}.
\end{proof}

\begin{corollary}[Uniform rank-one stability]\label{cor:rankoneStability}
For every \(L>0\) and compact \(J\Subset(-1,1)\),
\begin{equation}\label{eq:detstability}
 \limsup_{N\to\infty}
 \sup_{\substack{|u|\leq L\\t\in J}}
 \frac1N\log D_N(u,t)\leq0.
\end{equation}
\end{corollary}

\begin{proof}
Cofactor expansion gives, for every scalar \(b\),
\[
 |\det(K-be_1e_1^{\mathsf T})|
 \leq|\det K|+|b|\,|\det K^{(1)}|.
\]
Since \(\sup_{t\in J}|\beta a_p(t)|<\infty\), the assertion follows
from \eqref{eq:fullDetUniform} and \eqref{eq:minorDetUniform}.
\end{proof}

\begin{lemma}[Uniform energy-tail bound]\label{lem:energyTail}
For every compact \(J\Subset(-1,1)\), there is \(C<\infty\) such that,
for all \(N\geq3\), \(u\in\R\), and \(t\in J\),
\begin{equation}\label{eq:globalDetBound}
 \E|\det(K_N(u)-\beta a_p(t)e_1e_1^{\mathsf T})|
 \leq[C(1+|u|)]^{N-1}.
\end{equation}
Moreover,
\begin{equation}\label{eq:energyTail}
 \lim_{L\to\infty}\limsup_{N\to\infty}\frac1N
 \log\mathcal K_N(\{u:|u|>L\},J)=-\infty.
\end{equation}
\end{lemma}

\begin{proof}
For an \(n\times n\) matrix \(A\), the arithmetic--geometric mean
inequality for its singular values gives
\[
 |\det A|\leq\left(\frac{\|A\|_{\mathrm F}}{\sqrt n}\right)^n.
\]
Under \eqref{eq:GOEnormalization},
\[
 \|G_n\|_{\mathrm F}^2\stackrel d=\frac2n
 \chi^2_{n(n+1)/2},
 \qquad
 \left(\E\|G_n\|_{\mathrm F}^n\right)^{1/n}\leq C_0\sqrt n.
\]
Minkowski's inequality in \(L^n\), \eqref{eq:KNexact}, and the boundedness
of \(a_p\) on \(J\) prove \eqref{eq:globalDetBound}.

Also, \(0\leq v_p(t)\leq1\), \(|t^pu|\leq|u|\),
\(\int_J\nu_N(t)\,\dd t\leq1\), and
\(\kappa_{N,p}\leq e^{C_1N}\).  Hence \eqref{eq:exactMarkedKR} implies
\begin{align*}
 \mathcal K_N(\{|u|>L\},J)
 \leq e^{C_2N}\int_{|u|>L}
 \exp\left\{N\left(-\frac{u^2}{2}+\beta|u|
 +\log(1+|u|)\right)\right\}\dd u.
\end{align*}
For \(z\geq0\),
\[
 \beta z\leq\frac{z^2}{8}+2\beta^2,
 \qquad
 \log(1+z)\leq z\leq\frac{z^2}{8}+2.
\]
Consequently, for \(L>0\),
\begin{equation}\label{eq:explicitTail}
 \mathcal K_N(\{|u|>L\},J)
 \leq\frac4{NL}\exp\left\{N\left(C_3-\frac{L^2}{4}\right)\right\},
\end{equation}
which proves \eqref{eq:energyTail}.
\end{proof}

\begin{proposition}[Marked Kac--Rice bound]\label{prop:markedKR}
Fix \(p\in\{3,4,\ldots\}\), \(\beta>0\), and a compact interval
\(J\Subset(-1,1)\).  For every fixed closed interval \(I\subset J\)
of positive length,
\begin{equation}\label{eq:markedKR}
 \limsup_{N\to\infty}\frac1N\log\mathcal K_N(\R,I)
 \leq\sup_{u\in\R,\,t\in I}\PhiP(u,t).
\end{equation}
Since \(J\) is arbitrary, it follows in particular that, for every
\(q\in(0,1)\),
\begin{equation}\label{eq:markedband}
 \lim_{\eta\downarrow0}\limsup_{N\to\infty}
 \frac1N\log\E\MN(q,\eta)
 \leq\sup_{u\in\R}\PhiP(u,q).
\end{equation}
\end{proposition}

\begin{proof}
Fix \(L>0\).  Uniformly for \(|u|\leq L\) and \(t\in J\),
\begin{align*}
 \frac1N\log\rho_N(u)&=\SigmaP(u)+o(1),\\
 \frac1N\log\nu_N(t)&=\frac12\log(1-t^2)+o(1),\\
 \frac1N\log D_N(u,t)&\leq o(1),
\end{align*}
by \cref{lem:uniformGOEdet,cor:rankoneStability}, Stirling's formula,
and \eqref{eq:overlapdensity}.  Applying the elementary Laplace upper
bound to \eqref{eq:exactMarkedRatio} gives
\[
 \limsup_{N\to\infty}\frac1N\log\mathcal K_N([-L,L],I)
 \leq\sup_{\substack{|u|\leq L\\t\in I}}\PhiP(u,t).
\]
The energy-tail estimate \eqref{eq:energyTail} then removes the
truncation and proves \eqref{eq:markedKR}.

Finally, \(\Om(x)=O(\log(1+|x|))\) as \(|x|\to\infty\).  Thus, locally
uniformly for \(t\in(-1,1)\), the function \(\PhiP(u,t)\) tends to
\(-\infty\) quadratically as \(|u|\to\infty\).  It follows by a common
compact truncation that \(t\mapsto\sup_u\PhiP(u,t)\) is continuous.
Taking \(I=[q-\eta,q+\eta]\subset(-1,1)\) and then
\(\eta\downarrow0\) proves \eqref{eq:markedband}.
\end{proof}

\begin{remark}\label{rem:unionbound}
For every collection \(A\) of critical points,
\begin{equation}\label{eq:unionbound}
 Z_{N,\beta}\left(\bigcup_{x\in A}B(x,q,\eta)\right)
 \leq \sum_{x\in A}Z_{N,\beta}(B(x,q,\eta))
 \leq\MN(q,\eta).
\end{equation}
No separation assumption is needed for this upper bound.
\end{remark}

\section{A general-p marked transition}\label{sec:generalp}

The marked variational problem has a sharp transition throughout the
high-temperature side.  This calculation both implies the analytic
part of \cref{thm:main} and locates the obstruction for \(p\geq4\).
For \(p\geq3\), set
\begin{equation}\label{eq:generalThresholds}
 s_* = \frac{p-2}{p-1},
 \qquad
 u_{\infty,p}=2\sqrt{\frac{p-1}{p}},
 \qquad
 \beta_{\mathrm{sh}}(p)
 =\sqrt{\frac{(p-1)^{p-1}}{p(p-2)^{p-2}}}.
\end{equation}
For \(0<\beta\leq\beta_{\mathrm{sh}}(p)\), define
\begin{equation}\label{eq:DeltaGeneral}
 \Delta_{p,\beta}(s)
 =\sup_{u\in\R}\Phi_{p,\beta}(u,\sqrt s)-\frac{\beta^2}{2},
 \qquad 0\leq s<1.
\end{equation}

\begin{proposition}[General-\(p\) marked sign law]
\label{prop:general-p-marked}
Fix \(p\geq3\) and \(0<\beta\leq\beta_{\mathrm{sh}}(p)\).  Then
\begin{equation}\label{eq:generalSignLaw}
 \operatorname{sgn}\Delta_{p,\beta}(s)
 =\begin{cases}
  1, & 0\leq s<s_*,\\
  0, & s=s_*,\\
  -1, & s_*<s<1.
 \end{cases}
\end{equation}
At \(s=s_*\), the maximizing energy is unique and equals
\begin{equation}\label{eq:generalContactEnergy}
 u^*_{p,\beta}
 =\frac{2(p-1)}{p-2}\,\beta s_*^{p/2}.
\end{equation}
In particular,
\(u^*_{p,\beta_{\mathrm{sh}}(p)}=u_{\infty,p}\).
\end{proposition}

\begin{proof}
Direct integration in \eqref{eq:Omega} gives
\begin{equation}\label{eq:OmegaExplicit}
 \Om(x)=
 \begin{cases}
  \dfrac{x^2}{4}-\dfrac12, & |x|\leq2,\\[6pt]
  \dfrac{x^2}{4}-\dfrac12-
  \dfrac{|x|\sqrt{x^2-4}}{4}
  +\log\!\left(\dfrac{|x|+\sqrt{x^2-4}}2\right), & |x|\geq2.
 \end{cases}
\end{equation}
Consequently, on the quadratic branch,
\begin{equation}\label{eq:generalSigmaInside}
 \Sigma_p(u)
 =\frac12\log(p-1)-\frac{p-2}{4(p-1)}u^2,
 \qquad |u|\leq u_{\infty,p}.
\end{equation}
For \(u=u_{\infty,p}\cosh a\), \(a\geq0\), the same formula gives
\[
 -\Sigma_p'(u)
 =\frac{(p-2)\cosh a+p\sinh a}{\sqrt{p(p-1)}}.
\]
Together with \eqref{eq:generalSigmaInside}, this shows that
\(-\Sigma_p'\) is strictly increasing on \([0,\infty)\), with matching
one-sided values at \(u_{\infty,p}\).  By evenness, \(\Sigma_p\) is
strictly concave; it is also coercive by \eqref{eq:OmegaExplicit}.
Hence, after adding any linear mark, the maximizer is unique.
Since the coefficient of \(u\) in \(\Phi_{p,\beta}(u,\sqrt s)\) is
nonnegative, the unique maximizer is nonnegative.  Whenever it lies in
the quadratic branch, it is
\begin{equation}\label{eq:generalInnerOptimizer}
 u_{p,\beta}(s)
 =\frac{2(p-1)}{p-2}\,\beta s^{p/2},
\end{equation}
and substitution in \eqref{eq:PhiGeneral} gives
\begin{equation}\label{eq:general-p-inner}
 \Delta_{p,\beta}(s)
 =\frac12\log\bigl((p-1)(1-s)\bigr)
 +\frac{p\beta^2}{2}s^{p-1}
  \left(\frac{p-1}{p-2}s-1\right).
\end{equation}
The identity
\begin{equation}\label{eq:optimizerEdgeRatio}
 \frac{u_{p,\beta}(s)}{u_{\infty,p}}
 =\frac{\beta}{\beta_{\mathrm{sh}}(p)}
  \left(\frac{s}{s_*}\right)^{p/2}
\end{equation}
shows that \eqref{eq:general-p-inner} applies throughout
\(0\leq s\leq s_*\).  It immediately gives
\(\Delta_{p,\beta}(s_*)=0\) and \eqref{eq:generalContactEnergy}.

We next prove positivity for \(s<s_*\).  Put
\[
 n=p-2,
 \qquad
 x=\frac{s}{s_*}.
\]
At \(\beta=\beta_{\mathrm{sh}}(p)\), equation
\eqref{eq:general-p-inner} becomes
\begin{equation}\label{eq:generalInnerNormalized}
 2\Delta_{p,\beta_{\mathrm{sh}}(p)}(s)
 =\log\bigl(1+n(1-x)\bigr)-nx^{n+1}(1-x).
\end{equation}
For \(t=1-x\in(0,1)\),
\[
 \log(1+nt)>\frac{nt}{1+nt}>nt(1-t)^{n+1}.
\]
The second inequality follows from
\((1-t)^{-(n+1)}>1+nt\).  Thus the right-hand side of
\eqref{eq:generalInnerNormalized} is positive; the case \(t=1\) is
immediate from \(\log(1+n)>0\).  For \(s<s_*\), the
coefficient of \(\beta^2\) in \eqref{eq:general-p-inner} is negative,
and hence
\[
 \Delta_{p,\beta}(s)
 \geq\Delta_{p,\beta_{\mathrm{sh}}(p)}(s)>0.
\]

It remains to treat \(s>s_*\).  We first work at
\(\beta=\beta_{\mathrm{sh}}(p)\).  By
\eqref{eq:optimizerEdgeRatio}, the formal quadratic stationary point
then lies strictly beyond \(u_{\infty,p}\).  Strict concavity therefore
places the actual maximizer in the outer branch.  Write \(s=s_*x\) and
parametrize it by
\[
 u=u_{\infty,p}\cosh a,
 \qquad
 y=e^a.
\]
Formula \eqref{eq:OmegaExplicit} yields
\begin{equation}\label{eq:generalSigmaOutside}
 \Sigma_p(u_{\infty,p}\cosh a)
 =\frac12\log(p-1)-\frac{p-2}{2p}+a
  -\frac{p-1}{2p}e^{2a}+\frac1{2p}e^{-2a}.
\end{equation}
The critical-point equation for the maximizing \(a\) becomes
\begin{equation}\label{eq:generalOuterCritical}
 (n+1)y-y^{-1}=nx^{(n+2)/2}.
\end{equation}
Set
\[
 L=n+1-nx,
 \qquad
 z=x^{n/2}.
\]
Since \(1<x<(n+1)/n\), we have \(L>0\).  Envelope differentiation,
followed by substitution of \eqref{eq:generalOuterCritical}, gives
\begin{equation}\label{eq:general-p-outer-derivative}
 \frac{\dd}{\dd x}
 \Delta_{p,\beta_{\mathrm{sh}}(p)}(s_*x)
 =\frac{n}{2L}(y-zL)(zL-y^{-1}).
\end{equation}
For \(x>1\),
\[
 (zL)'=nz\left(\frac{L}{2x}-1\right)<0,
\]
so \(zL\) decreases from one.  Since \(y>1\), this gives
\(y-zL>0\).  To determine the other factor, define
\[
 f(v)=(n+1)v-v^{-1}-nx^{(n+2)/2}.
\]
The function \(f\) is strictly increasing, \(f(y)=0\), and a direct
calculation gives
\[
 zL\,f\bigl((zL)^{-1}\bigr)
 =(n+1)(1-x^nL)>0,
\]
because
\[
 (x^nL)'=n(n+1)x^{n-1}(1-x)<0,
\]
so \(x^nL\) decreases from one on this interval.  Hence
\(y<(zL)^{-1}\), so \(zL-y^{-1}<0\).  It follows from
\eqref{eq:general-p-outer-derivative} that
\[
 \Delta_{p,\beta_{\mathrm{sh}}(p)}(s)<0,
 \qquad s>s_*.
\]

Finally fix \(s>s_*\) and let \(u_{p,\beta}(s)\) be the maximizing
energy.  The envelope theorem gives
\begin{equation}\label{eq:generalBetaDerivative}
 \partial_\beta\Delta_{p,\beta}(s)
 =s^{p/2}\left[
  u_{p,\beta}(s)
  -\beta s^{(p-2)/2}\bigl(p-(p-1)s\bigr)
 \right].
\end{equation}
This derivative is strictly positive.  In the quadratic branch, the
claim reduces by \eqref{eq:generalInnerOptimizer} to \(s>s_*\).  In
the outer branch, \(u_{p,\beta}(s)\geq u_{\infty,p}\), while
\[
 s\longmapsto s^{(p-2)/2}\bigl(p-(p-1)s\bigr)
\]
is strictly decreasing for \(s>s_*\), and
\[
 2\beta_{\mathrm{sh}}(p)s_*^{(p-2)/2}=u_{\infty,p}.
\]
Since \(\beta\leq\beta_{\mathrm{sh}}(p)\), the bracket in
\eqref{eq:generalBetaDerivative} is again positive.  Therefore
\[
 \Delta_{p,\beta}(s)
 \leq\Delta_{p,\beta_{\mathrm{sh}}(p)}(s)<0,
\]
which completes the proof.
\end{proof}

By \cref{lem:fullFE}, the reference value \(\beta^2/2\) in
\eqref{eq:DeltaGeneral} is the full limiting free energy in this range.

\begin{proof}[Proof of \cref{thm:main}]
For \(q\leq2^{-1/2}\), \cref{prop:smallq} rules out positive
complexity.  For \(q>\sqrt{s_*}\), \cref{prop:general-p-marked} gives
a fixed gap
\[
 \gamma=\frac{\beta^2}{2}
 -\sup_{u\in\R}\Phi_{p,\beta}(u,q)>0.
\]
Continuity and coercivity give a fixed \(\eta_0>0\) such that
\[
 \sup_{\substack{u\in\R\\|t-q|\leq\eta_0}}
 \Phi_{p,\beta}(u,t)
 \leq\frac{\beta^2}{2}-\frac{3\gamma}{4}.
\]
After decreasing \(\eta_0\) if necessary so that
\([q-\eta_0,q+\eta_0]\Subset(0,1)\), \cref{prop:markedKR} implies that,
for all sufficiently large \(N\),
\[
 \E\MN(q,\eta_0)
 \leq\exp\left\{N\left(\frac{\beta^2}{2}-\frac\gamma2\right)\right\}.
\]
Markov's inequality and \eqref{eq:unionbound} therefore show, with
probability tending to one, that the union of any critical-point bands
of width at most \(\eta_0\) has free energy at most
\(\beta^2/2-\gamma/4\).  Since every admissible shattering width
\(\eta_N\) eventually satisfies \(\eta_N\leq\eta_0\), while
\cref{lem:fullFE} gives \(F_N(\beta)\to\beta^2/2\), this contradicts
condition~\textup{(4)} in \cref{def:shattering}.

Finally, when \(p=3\), one has
\(\sqrt{(p-2)/(p-1)}=2^{-1/2}\), so the two ranges in
\eqref{eq:mainRange} exhaust \(q\in(0,1)\).  Also
\(\beta_{\mathrm{sh}}(3)=\bsh\), equivalently
\(T_{\mathrm{sh}}(3)=\Tsh\), which proves the stated specialization.

The additional criterion \eqref{eq:mainCodeCriterion} and its
all-overlap consequence \eqref{eq:hotAllOverlap} are exactly
\cref{prop:hotAllOverlap}.
\end{proof}

\begin{remark}[The geometric window for \(p\geq4\)]
\label{rem:general-p-window}
The only overlap range not excluded by the small-overlap and marked
Kac--Rice arguments is
\begin{equation}\label{eq:generalWindow}
 \frac12<s\leq\frac{p-2}{p-1},
 \qquad\text{equivalently}\qquad
 \frac1{\sqrt2}<q\leq\sqrt{\frac{p-2}{p-1}}.
\end{equation}
For \(p=3\), this window is empty.  For \(p\geq4\), it is nonempty,
and at the contact point
\begin{equation}\label{eq:contactGeometryMargin}
 2s_*-1=\frac{p-3}{p-1}>0.
\end{equation}
Thus the thermodynamic contact occurs where the universal obstruction
based on negative overlaps no longer applies.  The positive margin
does not prove shattering: it only shows that the \(N+1\) cardinality
bound supplies no control there.  The spherical-code argument in
\cref{prop:hotAllOverlap} resolves every point in this window for which
\begin{equation}\label{eq:windowCodeCriterion}
 \beta^2<
 -2\mathcal R_{\mathrm{LP}}(2q^2-1)-\log(1-q^2).
\end{equation}
In particular, it resolves the whole window for
\(\beta\leq\sqrt{\log2}\).  Resolving the remaining part requires more
than one-point marked Kac--Rice, universal latitude geometry, and the
entropy estimate used here.

If \(r=0\), condition~\textup{(2)} in \cref{def:shattering} itself
forces strictly negative pairwise center overlaps.  Then
\cref{lem:obtuse} gives at most \(N+1\) centers for every \(p\), every
temperature, and every \(q\).  Thus the unresolved window concerns
positive separation parameters \(r\).
\end{remark}

\section{Endpoint optimization for the pure 3-spin model}\label{sec:optimization}

Although \cref{prop:general-p-marked} already implies the endpoint
inequality, an explicit \(p=3\) optimization identifies the order of
the tangency.  We now specialize to \(p=3\) and
\(\beta=\bsh=2/\sqrt3\).  Set
\begin{equation}\label{eq:uinfty}
 u_\infty=2\sqrt{\frac{2}{3}}=\sqrt{\frac83}.
\end{equation}
Specializing \eqref{eq:Sigma} and using \eqref{eq:OmegaExplicit},
\begin{equation}\label{eq:Sigma3}
 \Sigma_3(u)
 =\frac12+\frac12\log2-\frac{u^2}{2}
  +\Om\left(\sqrt{\frac32}\,u\right).
\end{equation}
In particular, for \(0\leq u\leq u_\infty\),
\begin{equation}\label{eq:SigmaInside}
 \Sigma_3(u)=\frac12\log2-\frac{u^2}{8}.
\end{equation}

Write \(s=q^2\).  From \eqref{eq:PhiGeneral},
\begin{equation}\label{eq:Psi}
 \Psi(u,s):=\Phi_{3,\bsh}(u,\sqrt s)
 =\Sigma_3(u)
  +\frac2{\sqrt3}s^{3/2}u
  +\frac12\log(1-s)
  +\frac23(1-3s^2+2s^3).
\end{equation}

\begin{proposition}[Endpoint variational inequality]\label{prop:variational}
For every \(s\in[1/2,1)\),
\begin{equation}\label{eq:variational}
 \sup_{u\in\R}\Psi(u,s)\leq\frac23.
\end{equation}
Equality in \eqref{eq:variational} holds if and only if \(s=1/2\); at
this value, the unique maximizing energy is \(u=u_\infty\).  In
particular, for every \(q\in(2^{-1/2},1)\),
\begin{equation}\label{eq:strictgap}
 \sup_{u\in\R}\Phi_{3,\bsh}(u,q)<\frac23.
\end{equation}
\end{proposition}

\begin{proof}
Since \(\Sigma_3\) is even and the coefficient of \(u\) in \eqref{eq:Psi} is positive, it suffices to maximize over \(u\geq0\).

For \(0\leq u\leq u_\infty\), \eqref{eq:SigmaInside} gives
\begin{equation}\label{eq:derivativeInside}
 \partial_u\Psi(u,s)
 =-\frac u4+\frac2{\sqrt3}s^{3/2}.
\end{equation}
At \(u=u_\infty\), this derivative is nonnegative for \(s\geq1/2\), with equality only at \(s=1/2\).  Hence no maximizer lies in \([0,u_\infty)\).

For \(u\geq u_\infty\), write
\begin{equation}\label{eq:hyperbolic}
 u=u_\infty\cosh a,
 \qquad a\geq0.
\end{equation}
Since \(\sqrt{3/2}\,u=2\cosh a\), \eqref{eq:OmegaExplicit} yields
\begin{equation}\label{eq:SigmaHyperbolic}
 \Sigma_3(u_\infty\cosh a)
 =\frac12\log2-\frac16+a-\frac13e^{2a}+\frac16e^{-2a}.
\end{equation}
Define
\[
 D(a,s)=\Psi(u_\infty\cosh a,s)-\frac23.
\]
Substituting \eqref{eq:SigmaHyperbolic} into \eqref{eq:Psi}, we obtain
\begin{align}
 D(a,s)
 ={}&\frac12\log\bigl(2(1-s)\bigr)-\frac16+a
 -\frac13e^{2a}+\frac16e^{-2a}\notag\\
 &+\frac{4\sqrt2}{3}s^{3/2}\cosh a
 -2s^2+\frac43s^3.
 \label{eq:D}
\end{align}
Let \(y=e^a\).  Differentiation gives the exact factorization
\begin{equation}\label{eq:Dfactor}
 \partial_aD(a,s)
 =-\frac{(y-1)(y+1)\bigl(2y^2-2\sqrt2s^{3/2}y-1\bigr)}{3y^2}.
\end{equation}
For \(s>1/2\), \(D(\cdot,s)\) first increases and then decreases on \([0,\infty)\).  Its unique maximizer \(a_s>0\) is determined by
\begin{equation}\label{eq:criticala}
 2e^{2a_s}-2\sqrt2s^{3/2}e^{a_s}-1=0.
\end{equation}
At \(s=1/2\), the maximizing point is \(a_{1/2}=0\).  Equivalently,
\begin{equation}\label{eq:criticalhyperbolic}
 \cosh a_s+3\sinh a_s=2\sqrt2s^{3/2}.
\end{equation}
Let
\begin{equation}\label{eq:Gdef}
 G(s)=D(a_s,s),
 \qquad s\in[1/2,1).
\end{equation}
A direct substitution into \eqref{eq:D} gives
\begin{equation}\label{eq:Ghalf}
 G(1/2)=0.
\end{equation}
For \(s\in(1/2,1)\), the envelope theorem and \eqref{eq:D} give
\begin{equation}\label{eq:Gprime1}
 G'(s)
 =2\sqrt2\,s^{1/2}\cosh a_s+4s^2-4s-\frac1{2(1-s)}.
\end{equation}
Solving \eqref{eq:criticala} for \(e^{a_s}\) and then for \(\cosh a_s\) gives
\begin{equation}\label{eq:coshas}
 \cosh a_s
 =\frac{3\sqrt{1+s^3}-s^{3/2}}{2\sqrt2}.
\end{equation}
Therefore
\begin{equation}\label{eq:Gprime2}
 G'(s)
 =3\sqrt{s+s^4}+3s^2-4s-\frac1{2(1-s)}.
\end{equation}
Set
\begin{equation}\label{eq:Rdef}
 \mathcal R(s)=4s-3s^2+\frac1{2(1-s)}.
\end{equation}
Then \(G'(s)=3\sqrt{s+s^4}-\mathcal R(s)\).  An exact algebraic identity is
\begin{equation}\label{eq:squarefactor}
 \mathcal R(s)^2-9(s+s^4)
 =-\frac{(2s-1)^3(12s^2-14s+1)}{4(1-s)^2}.
\end{equation}
For \(s\in(1/2,1)\),
\[
 12s^2-14s+1<0,
\]
because its roots are \((7\pm\sqrt{37})/12\), with the upper root larger than one.  Thus the right-hand side of \eqref{eq:squarefactor} is strictly positive.  Since \(\mathcal R(s)>0\),
\[
 \mathcal R(s)>3\sqrt{s+s^4},
\]
and hence
\begin{equation}\label{eq:Gnegative}
 G'(s)<0,
 \qquad s\in(1/2,1).
\end{equation}
Together with \eqref{eq:Ghalf}, this proves \(G(s)<0\) for every \(s>1/2\).  Equality occurs only at \(s=1/2\), where \(a=0\), equivalently \(u=u_\infty\).
\end{proof}

\begin{remark}[Tangency at the geometric threshold]\label{rem:tangency}
The equality point in \cref{prop:variational} is
\[
 q=\frac1{\sqrt2},
 \qquad
 u=u_\infty.
\]
This is exactly the boundary at which \cref{lem:latitude} changes from requiring negative pairwise center overlaps to permitting nonnegative ones.  Thus the thermodynamic exponent reaches the full free energy only at the point where exponential disjointness is geometrically impossible.

Writing
\[
 \gamma(q)=\frac23-\sup_{u\in\R}\Phi_{3,\bsh}(u,q),
\]
the contact is fourth order.  Indeed, Taylor expansion of
\eqref{eq:Gprime2} at \(s=1/2\) gives, as \(q\downarrow2^{-1/2}\),
\begin{equation}\label{eq:gapExpansion}
 \gamma(q)=\frac43\left(q^2-\frac12\right)^4
 +O\left(\left(q^2-\frac12\right)^5\right)
 =\frac{16}{3}\left(q-\frac1{\sqrt2}\right)^4
 +O\left(\left(q-\frac1{\sqrt2}\right)^5\right).
\end{equation}
This slow closing of the gap is another reason that the fixed-\(q\)
quantifier in \cref{def:shattering} matters.
\end{remark}

\section{Discussion and scope}\label{sec:discussion}

For every \(0<\beta\leq\beta_{\mathrm{sh}}\), the marked exponent
touches the full free energy at \(q^2=1/2\) and is strictly below it for
every \(q^2>1/2\).  The deterministic obstruction rules out positive
complexity at and below the contact latitude.  At
\(\beta=\beta_{\mathrm{sh}}\), the optimizing energy reaches
\(u_\infty\) and the contact becomes fourth order, as shown in
\cref{rem:tangency}.  Thus the pure \(3\)-spin landscape is not
shattered throughout \(T\geq T_{\mathrm{sh}}\).

The joint proof of Theorems~2.5 and~8.4 in
\cite[Section~5]{BenArousJagannath2024} assumes the strict inequality
\(\beta>\beta_{\mathrm{sh}}\).  In the \(p=3\) discussion in
\cite[Section~8]{BenArousJagannath2024}, this yields
\[
 q_{**}(\beta)>\frac1{\sqrt2},
\]
which provides the positive geometric margin used in the disjointness
argument.  At equality,
\[
 q_{**}(\beta_{\mathrm{sh}})=\frac1{\sqrt2},
\]
and the deterministic argument in \cref{sec:geometry} shows that an exponentially large disjoint family is impossible.  The marked Kac--Rice calculation shows, in addition, that moving to any fixed larger \(q\) loses a strictly positive amount of free energy.  Thus the strict inequality in the proof reflects a real endpoint transition.

The coincidence between the marked transition and the geometric cutoff
is special to \(p=3\).  For general \(p\), write \(q_{\mathrm{cont}}\)
for the contact latitude.  Then
\[
 q_{\mathrm{cont}}^2=s_* =\frac{p-2}{p-1},
\]
so for \(p\geq4\),
\[
 2q_{\mathrm{cont}}^2-1=\frac{p-3}{p-1}>0.
\]
Hence the present argument leaves the range
\(1/2<q^2\leq s_*\) unresolved only where the reverse weak inequality
to \eqref{eq:windowCodeCriterion} holds.  This remaining set is
contained in \(\sqrt{\log2}<\beta\leq\beta_{\mathrm{sh}}(p)\); for
\(\beta\leq\sqrt{\log2}\), \cref{prop:hotAllOverlap} supplies an
all-overlap obstruction.  To our knowledge, the literal
dynamical endpoint for \(p\geq4\) remains open under
\cref{def:shattering}; \cref{thm:main} records the ranges that the
present method does settle.

This paper settles fixed-overlap, critical-point-band non-shattering at
and above the dynamical temperature for the pure \(3\)-spin model.  It
does not settle the strict interior \(\Ts<T<\Tsh\), the lower endpoint
\(T=\Ts\), or the lower-temperature regime.  Accordingly, any \(p=3\)
shattered regime under \cref{def:shattering} must lie strictly below
\(\Tsh\).

The terminology and geometric definition used here should also be
distinguished from their physical antecedents.  Kirkpatrick and
Thirumalai's early analysis of mean-field \(p\)-spin models identified a
dynamical arrest above the static transition and related it to broken
replica symmetry \cite{KirkpatrickThirumalai1987Dynamics,KirkpatrickThirumalai1987Connections};
their metastable-state picture further tied the appearance of such
states to the loss of ergodicity \cite{KirkpatrickThirumalai1988Metastable}.
These works are foundational precursors of the landscape-fragmentation
scenario addressed by modern shattering results.  They do not, however,
use the fixed-overlap band criterion in \cref{def:shattering}; the
present theorem concerns that particular rigorous formulation.

Recent works study related but different notions.  For the Ising pure
\(p\)-spin model at large \(p\), Gibbs-measure shattering was proved in
\cite{GamarnikJagannathKizildag2025} and extended down to the
conjectural dynamical scale, to leading order as \(p\to\infty\), in
\cite{ElAlaoui2026}.  Those results use configuration-space clusters
and overlap-gap methods rather than the fixed-overlap critical-point
bands of \cref{def:shattering}.

El Alaoui,
Montanari, and Sellke \cite{ElAlaouiMontanariSellke2025} prove a
Gibbs-measure shattering decomposition for pure spherical models when
\(p\) is sufficiently large.  Their clusters need not be bands around
critical points, so that result neither covers \(p=3\) nor decides the
literal endpoint question under \cref{def:shattering}.

The Franz--Parisi potential was introduced in
\cite{FranzParisi1995}.  Auffinger, El Alaoui, and Sellke
\cite[Section~5]{AuffingerElAlaouiSellke2025} discuss the strict
spherical criterion
\[
 \beta^2\xi'(s)>\frac{s}{1-s},\qquad s\in(0,1).
\]
For the pure model \(\xi(s)=s^p\), this becomes
\[
 \beta^2p s^{p-2}(1-s)>1.
\]
For every \(0<\beta\leq\beta_{\mathrm{sh}}(p)\), the reverse weak
inequality holds for all \(s\in(0,1)\); equality occurs only when
\(\beta=\beta_{\mathrm{sh}}(p)\) and
\(s=(p-2)/(p-1)\).  For \(p=3\), this is \(s=1/2=q^2\), matching the
contact point in \cref{prop:general-p-marked}.  This remains only a
consistency check: equivalence between that Franz--Parisi-potential
criterion and \cref{def:shattering} is not presently established.

\section*{Acknowledgments}

The author is deeply grateful to G\'erard Ben Arous for many helpful
discussions and for his generous guidance on spherical spin glasses,
shattering, and many related topics.  The author also thanks Theodore
R.~Kirkpatrick for many helpful comments. This work was supported by Jang Young Sil
Fellow Program at KAIST.

\appendix

\section{Gaussian regression at a critical point}\label{app:regression}

For completeness, we verify \cref{lem:finitecritical} and
\eqref{eq:condmean}--\eqref{eq:rankonecov}.  By rotational invariance, take
\[
 x=\sqrt N e_N,
 \qquad
 \sigma=\sqrt N\bigl(te_N+\sqrt{1-t^2}\,e_1\bigr).
\]
View \(H_N\) as the restriction to \(\SN\) of the homogeneous Gaussian
polynomial on \(\R^N\) with covariance
\[
 \E H_N(z)H_N(w)=N^{1-p}(z\cdot w)^p.
\]
For tangent indices \(1\leq i,j\leq N-1\), let
\(B_{ij}=\partial^{\mathrm{amb}}_{ij}H_N(x)\).  Homogeneity and the
second fundamental form of \(S^{N-1}(\sqrt N)\) give
\begin{equation}\label{eq:sphericalHessianAmbient}
 (\nabla^2H_N(x))_{ij}
 =B_{ij}-\frac pN H_N(x)\delta_{ij}.
\end{equation}
Differentiating the ambient covariance shows
\begin{align}
 \operatorname{Cov}(B_{ij},B_{k\ell})
 &=\frac{p(p-1)}N
 (\delta_{ik}\delta_{j\ell}+\delta_{i\ell}\delta_{jk}),
 \label{eq:ambientHessianCov}\\
 \operatorname{Cov}(B_{ij},H_N(x))
 &=\operatorname{Cov}(B_{ij},\partial_kH_N(x))=0.
 \label{eq:ambientHessianIndependent}
\end{align}
Hence \(B\) is independent of \((H_N(x),\nabla H_N(x))\), and under
\(H_N(x)=-Nu\), \eqref{eq:sphericalHessianAmbient} is a GOE matrix with
mean \(puI_{N-1}\) and covariance \eqref{eq:conditionalHessianCov}.
With the normalization \eqref{eq:GOEnormalization}, this is exactly
\(K_N(u)\) from \eqref{eq:KNexact}, proving \cref{lem:finitecritical}.

Let \(\partial_i\) denote differentiation at \(x\) in geodesic normal coordinates in the tangent direction \(e_i\), \(1\leq i\leq N-1\).  Differentiating
\[
 C(z,w)=\E H_N(z)H_N(w)=N R(z,w)^p
\]
gives
\begin{align}
 \operatorname{Var}(H_N(x))&=N,\label{eq:cov1}\\
 \operatorname{Cov}(H_N(x),H_N(\sigma))&=Nt^p,\label{eq:cov2}\\
 \operatorname{Cov}(\partial_iH_N(x),\partial_jH_N(x))&=p\delta_{ij},\label{eq:cov3}\\
 \operatorname{Cov}(H_N(\sigma),\partial_iH_N(x))
 &=p t^{p-1}\sigma_i.\label{eq:cov4}
\end{align}
Only the \(i=1\) term in \eqref{eq:cov4} is nonzero.  Since \(H_N(x)\) is independent of \(\nabla H_N(x)\), Gaussian regression immediately yields
\[
 \E[H_N(\sigma)\mid H_N(x)=-Nu,\nabla H_N(x)=0]
 =-Nt^pu,
\]
and
\begin{align*}
 \operatorname{Var}(H_N(\sigma)\mid H_N(x),\nabla H_N(x))
 &=N-Nt^{2p}-\frac1p\sum_{i=1}^{N-1}
 \operatorname{Cov}(H_N(\sigma),\partial_iH_N(x))^2\\
 &=N\bigl(1-t^{2p}-pt^{2p-2}(1-t^2)\bigr),
\end{align*}
which is \eqref{eq:condvar}.

For the Riemannian Hessian, a second differentiation gives
\begin{equation}\label{eq:hesscovuncond}
 \operatorname{Cov}(\partial_{ij}^2H_N(x),H_N(\sigma))
 =p(p-1)t^{p-2}\frac{\sigma_i\sigma_j}{N}
  -pt^p\delta_{ij}.
\end{equation}
The curvature term is the second term on the right.  Similarly,
\begin{equation}\label{eq:hessvaluecov}
 \operatorname{Cov}(\partial_{ij}^2H_N(x),H_N(x))=-p\delta_{ij}.
\end{equation}
The Hessian is independent of the gradient.  Conditioning on \(H_N(x)\) therefore cancels the identity term in \eqref{eq:hesscovuncond}:
\begin{align*}
 &\operatorname{Cov}(\partial_{ij}^2H_N(x),H_N(\sigma)
 \mid H_N(x),\nabla H_N(x))\\
 &\quad=p(p-1)t^{p-2}\frac{\sigma_i\sigma_j}{N}.
\end{align*}
Since \(\sigma_1^2/N=1-t^2\) and \(\sigma_i=0\) for \(2\leq i\leq N-1\), this is precisely \eqref{eq:rankonecov}.

\end{document}